\documentclass[11pt,twoside]{amsart}

\usepackage[T1]{fontenc}
\usepackage[english]{babel}
\usepackage{amssymb,amsfonts,amsthm,amsmath,mathrsfs}
\usepackage{layout,version,enumerate}
\usepackage[all,matrix]{xy}
\usepackage[varg]{pxfonts}
\usepackage{indentfirst}
\usepackage{url}
\usepackage[colorlinks=true,linkcolor=blue,citecolor=magenta]{hyperref}
\usepackage[numbers]{natbib}

\theoremstyle{plain}
\newtheorem{thm}{Theorem}[section]
\newtheorem{prop}[thm]{Proposition}
\newtheorem{crl}[thm]{Corollary}
\newtheorem{lem}[thm]{Lemma}

\newtheorem{thmi}{Theorem}
\newtheorem{crli}[thmi]{Corollary}
\newtheorem{propi}[thmi]{Proposition}

\theoremstyle{definition}
\newtheorem{dfn}[thm]{Definition}

\theoremstyle{remark}
\newtheorem{exm}[thm]{Example}
\newtheorem{rmk}[thm]{Remark}

\newcommand{\bn}{\mathbf{N}}
\newcommand{\bz}{\mathbf{Z}}
\newcommand{\br}{\mathbf{R}}
\newcommand{\fx}{\mathfrak{X}}

\newcommand{\fa}{\mathfrak{A}}

\newcommand{\ca}{\mathcal{A}}
\newcommand{\cb}{\mathcal{B}}

\newcommand{\rb}{\mathscr{B}}
\newcommand{\rp}{\mathscr{P}}

\newcommand{\ul}{\mathfrak{u}}

\newcommand{\ph}{\varphi}
\newcommand{\eps}{\varepsilon}

\newcommand{\aut}{\mathrm{Aut}}

\renewcommand{\bar}[1]{\overline{#1}}

\renewcommand{\d}{\mathrm{d}}

\newcommand{\act}{\operatorname{Act}}
\newcommand{\bact}{\overline{\operatorname{Act}}}
\renewcommand{\part}{\operatorname{Part}}
\newcommand{\malg}{\operatorname{MAlg}}
\newcommand{\coeff}{\mathfrak c}
\newcommand{\fix}{\operatorname{Fix}}
\renewcommand{\hom}{\operatorname{Hom}}
\newcommand{\PSL}{\operatorname{PSL}}
\newcommand{\heb}{\mathrm{h}}

\long\def\/*#1*/{}

\begin{document}

\thispagestyle{empty}
\title{Ultraproducts, weak equivalence and sofic entropy}
\author{Alessandro Carderi}
\date{\today}

\begin{abstract}
In this work, we study pmp actions of countable groups on arbitrary diffuse probability spaces under the point of view of weak equivalence. We will show that any such an action is weakly equivalent to an action on a standard probability space. We also propose a metric on the space of actions modulo weak equivalence which is equivalent to the topology of Ab\'{e}rt and Elek. We will give a simpler proof of the compactness of the space, showing that convergence is characterized by ultraproducts. Using this topology, we will show that a profinite action is weakly equivalent to an ultraproduct of finite actions. 

Finally, combining our results with another result of Ab\'{e}rt and Elek, we will obtain a corollary about sofic entropy. We will show that for free groups and some property (T) groups, sofic entropy of profinite actions depends crucially on the chosen sofic approximation.
\end{abstract}

\maketitle

\section*{Introduction}

Measure preserving actions of countable groups on standard probability spaces have been studied for more than a century. Recently there has been some interest in ultraproduct of actions and their connection with sofic groups, see for example \cite{Conley2013a}, \cite{Abert2011c}, \cite{Pestov2008}, \cite{Elek2005} and \cite{Kerr2013b}. Ultraproducts are a natural limit procedure and the measure preserving actions constructed in this way, remember many properties of the sequences of actions used in their construction. One of the main difficulties of this construction, is that ultraproduct actions are defined on the \textit{Loeb probability space}, which is isomorphic as a measure space to $\{0,1\}^\br$ equipped with the product measure (Theorem \ref{thm:maharamultra}).

Some of the theory of measure preserving actions on standard probability spaces easily generalizes to general measure spaces. For example Dye worked without any assumption on the probability space in \cite{Dye1959}. Anyway not much is known in the general setting. In this work, we will study actions on general probability space under the point of view of \textit{weak containment}. We say that an action $a$ of the group $G$ on the probability space $(\fx_a,\mu_a)$ is weakly contained in an action $b$ on the probability space $(\fx_b,\mu_b)$ if for every $\eps>0$, for every finite partition $\alpha=\{A_1,\ldots,A_n\}$ of $\fx_a$ and for every finite subset $F$ of the group $G$, there exists a finite partition $\beta=\{B_1,\ldots,B_n\}$ of $\fx_b$ such that \[\sum_{i,j\leq n}\sum_{f\in F} |\mu_a(A_i\cap fA_j)-\mu_b(B_i\cap fB_j)|<\eps.\]

We can interpret the vector $(\mu_a(A_i\cap fA_j))_{i,j,f}$ as the $F$-\textit{statistics} of the action $a$ on $\alpha$ and an action $a$ is weakly contained in an action $b$ if we can approximate the statistics of $a$ with partitions in $b$. We say that two actions are \textit{weakly equivalent} if they are weakly contained one into the other. The definition of weak containment in the context of standard probability spaces was introduced by Kechris in \cite{Kechris2010} and the same definition makes sense for actions on arbitrary probability spaces. We will prove the following.

\begin{thmi}\label{thma}
  Every probability measure preserving action of a countable group on a diffuse space is weakly equivalent to an action on a standard probability space. 
\end{thmi}

More precisely, we will prove in Theorem \ref{thm:thma} that every probability measure preserving action on a diffuse space has a standard diffuse factor which is weakly equivalent to the action. 

Theorem \ref{thma} also implies that the family of weakly equivalence classes of actions on finite or diffuse probability spaces is a set and it is isomorphic to the set of classes of actions on $\{1,\ldots,n\}$ for $n\in \bn$ and on a fixed standard probability space, say $[0,1]$ with respect to the Lebesgue measure. Let us denote the set of classes by $\bact(G)$. One of the avantages of working with $\bact(G)$ is that it is closed under ultraproducts: for every sequence of (classes of) actions $(a_n)_n$ of $\bact(G)$ and for every ultrafilter $\ul$, the (class of the) action on the ultraproduct space $a_\ul$ is still an element of $\bact(G)$. 

Ab\'{e}rt and Elek defined in \cite{Abert2011c} a compact, metric topology on the space of weak equivalence classes of actions on a standard Borel space which, by Theorem \ref{thma}, is isomorphic to $\bact(G)$. Once we identify these two spaces, it is not hard to see that every converging sequence converges to the class of its ultraproduct (with respect to any ultrafilter). Since ultraproducts of sequences of actions always exist, the topology is necessarily compact and it is completely determined by this property. This compact space was later studied in \cite{Tucker2012}, \cite{Burton2015} and \cite{Burton2015a}.

We introduce in Definition \ref{dfn:wct} a compact, metric topology on $\bact(G)$, which is equivalent to the topology of Ab\'{e}rt and Elek. The metric is essentially the metric used in \cite{Burton2015}. A sequence is converging for this topology if the asymptotic of the \textit{statistics} of the actions converges to the \textit{statistics} of the limit action and as in the case of Ab\'{e}rt and Elek's topology, every converging sequence converges to its ultraproduct, see Theorem \ref{thm:compact}. 

One of the aims of this work is to give a concise, simple and self-contained proof of the compactness of the space, Theorem 1 of \cite{Abert2011c}. 

It will follow easily from the definition of the topology on $\bact(G)$, that if $\{H_n\}_n$ is a descending chain of finite index subgroups of $G$, then the limit of the sequence of the finite actions $G/H_n$ is the (class of the) profinite action $a^{(H_n)}$. Since limits are always weakly equivalent to the ultraproducts of the sequences, we get the following interesting corollary. 

\begin{crli}\label{crla}
  Let $G$ be a countable group and let $(H_n)$ be a chain of finite index subgroups. Then the profinite action $a^{(H_n)}$ associated to the sequence  $(H_n)_n$ is weakly equivalent to the ultraproduct of the sequence of finite actions on the quotients $(G/H_n)$ with respect to any ultrafilter. 
\end{crli}

We will give an application of Corollary \ref{crla} in the context of sofic entropy. 

\subsection*{Sofic Entropy}

The entropy of a dynamical system was introduced by Kolmogorov in the fifties for actions of the integer group on a probability space. This invariant has been fundamental for distinguish unitarily equivalent actions. For instance Ornstein was able to classify Bernoulli shifts of the integer group: two such actions are conjugate if and only if the base spaces have the same entropy. The theory was successfully extended to actions of amenable groups and Ornstein and Weiss were able to distinguish Bernoulli shifts over base space with different entropy. While the entropy of actions of amenable groups was widely studied, there were some evidences pointing out that it would not have been possible to extend the definition to non amenable groups: some of the crucial properties of the entropy can not be true for the entropy of actions of such groups. In particular the question about Bernoulli shifts was unsolved. 
Several years later in 2010, Bowen in \cite{Bowen2010} and \cite{Bowen2010b} introduced a new concept of entropy for action of sofic groups which extends the previous definition in the amenable case. Using this new entropy, he was able to distinguish Bernoulli shifts of a large class of sofic groups as for amenable groups: the entropy of the base space is an invariant. This classification was extended to all sofic groups shortly later by Kerr and Li in \cite{Kerr2011a}, where they also proposed a definition of sofic topological entropy and stated a variational principle.

One of the major differences between entropy theory of amenable groups and sofic groups, is that the definition of entropy for sofic groups depends on a fixed sofic approximation. Once the approximation is fixed, the entropy is only defined (as a non-negative number) for some actions, which we will call its \textit{domain of definition}. For the others the entropy is just declared to be $-\infty$. This means that each sofic approximation gives us a possibly different notion of entropy which has its proper domain. Bowen proved in \cite{Bowen2010}, see also \cite{Kerr2013a}, that for Bernoulli shifts the entropy is always defined and its value does not depend on the sofic approximation. This phenomenon was later extended to algebraic actions see \cite{Bowen2011}, \cite{Kerr2011a} and \cite{Hayes2014a}. 

We will try to clarify how the domain of definition of sofic entropy depends on the sofic approximation. The answer appears extremely simple when the sofic entropy is defined using a sofic approximation which comes from a chain of finite index subgroups. In fact if $G$ is a residually finite group and $(H_n)_n$ is a chain of subgroups such that the associated profinite action is free, then the sequence of actions of $G$ on the finite quotients is a sofic approximation of $G$, which we will denote by $\Sigma_{(H_n)}$. The following proposition is a consequence of Corollary \ref{crla}.

\begin{propi}
  Let $G$ be a residually finite group and let $(H_n)_n$ be a chain of finite index subgroups such that the associated profinite action  $a^{(H_n)}$ is free. Then for every measure preserving action $b$ of $G$ on a standard probability space $(\fx,\mu)$, we have that $\heb_{\Sigma_{(H_n)}}(b)>-\infty$ if and only if the action $b$ is weakly contained in the profinite action $a^{(H_n)}$. 
\end{propi}

The proposition tells us that the domains of definition depend on the sofic approximation: there are actions that are in some domains but not in others. Ab\'{e}rt and Elek in \cite{Abert2012b} proved an interesting result about rigidity of weak equivalence for profinite actions, which we can combine with the previous proposition to get the following result. 

\begin{thmi}\label{thmc}
  Let $G$ be a countable free group or $\PSL_k(\bz)$ for $k\geq 2$. Then there is a continuum of normal chains $\{(H^r_n)_n\}_{r\in\br}$ such that $\heb_{\Sigma_{(H^r_n)}}(a^{(H^s_n)})>-\infty$ if and only if $r=s$.
\end{thmi}

Observe that the entropy of profinite actions has been calculated in \cite{Chung2014} and it is always $0$, when it is defined. Since profinite actions have a generating partition with finite (actually arbitrarily small) entropy (Lemma \ref{lem:genprof} ), we can use Bowen's computation of entropy for products of actions with Bernoulli shifts \cite{Bowen2010} to get actions which have positive entropy with respect to some sofic approximations and $-\infty$ with respect to others, see Theorem \ref{thm:randinf}.

We do not know any action for which the sofic entropy can have two different non-negative values. 

\subsection*{Acknowledgements }
This work is part of the author's thesis under the supervision of Damien Gaboriau at the ENS de Lyon. The author is very grateful to D. Gaboriau for all the support through these years and in particular the author warmly thanks him for the important comments about this work and for carefully reading a first version. The author also wants to thank Lewis Bowen and Robin Tucker-Drob for various discussions. The author was supported by the ANR project GAMME (ANR-14-CE25-0004).

\section{Ultrapoducts of probability spaces}

In this section, we describe the ultraproduct of probability measure spaces. These probability spaces were introduced by Loeb in \cite{Loeb1975} in the language of non-standard analysis and they are often called Loeb spaces. All the material presented here is well-known and a recent exposition can be found in \cite{Conley2013a} and \cite{Elek2012b}. 

Let us fix a non-principal ultrafilter $\ul$ on $\bn$. 

\subsection{Set-theoretic ultraproducts}

\begin{dfn}
  Let $\{X_n\}_{n\in\bn}$ be a family of sets and let $X$ be their product $X:=\prod_{n\in\bn} X_n.$ We define the \textbf{ultraproduct} of the family $\{X_n\}_n$ to be the following quotient of $X$ \[X_\ul:=X/\sim_\ul\quad\text{ where }\quad (x_n)_n \sim_\ul (y_n)_n\ \text{ if }\ \{n : x_{n}= y_{n}\}\in \ul.\]
\end{dfn}

We will denote by $x_\ul$ and $A_\ul$ elements and subsets of $X_\ul$. For a sequence $(x_n)_n\in X$, we will denote by $[x_n]_\ul$ its class in $X_\ul$ and similarly for a sequence of subsets $\{A_n\subset X_n\}_n$, we will denote by $[A_n]_\ul$ the class of $(A_n)_n$. 

It is easy to observe that \[[A_n]_\ul \cap [B_n]_\ul=[A_n\cap B_n]_\ul,\quad [A_n]_\ul\cup [B_n]_\ul=[A_n\cup B_n]_\ul.\]  

\begin{rmk}\label{rmk:card}
We remark that if $\{X_n\}_n$ is a sequence of finite sets such that $\lim_\ul |X_n|=\infty$ or if it is a sequence of countable non-finite sets, the ultraproduct $X_\ul$ has the cardinality of the continuum. 

In fact, it is easy to construct a surjective map from $X_\ul$ to interval $[0,1]$. For example, if $X_n=\{1,\ldots,n\}$ then the map can be defined as \[\ph:X_\ul\to [0,1],\quad \ph([a_n]_n)=:\lim_{n\in \ul} \frac{a_n}{n},\] where the limit on the right is the limit with respect to the Euclidean topology. Since the rationals are dense in the interval, the map $\ph$ has to be surjective and a similar argument works for the general case. 
\end{rmk}

\subsection{Metric ultraproducts}

\begin{dfn}
  Let $\{(M_n,\d_n)\}_{n\in\bn}$ be a family of uniformly bounded metric spaces. We define the pseudo-metric $\d_\ul$ on $M:=\prod_{n\in \bn}M_n$ by \[\d_\ul((x_n)_n,(y_n)_n):=\lim_{n\in\ul}\d_n(x_n,y_n).\]

  We define the \textbf{metric ultraproduct} of the family $\{(M_n,\d_n)\}_n$ with respect to the ultrafilter $\ul$ to be the metric space associated to the pseudo-metric $\d_\ul$, that is $M_\ul:= M/\{\d_\ul=0\}.$
\end{dfn}


\begin{rmk}
  Let $\{G_n\}_n$ be a sequence of groups and let $\d_n$ be a bounded bi-invariant metric on $G_n$. It is easy to check that the subgroup \[K_\ul:=\left\lbrace (g_n)_n\in \prod_n G_n:\ \d_\ul((g_n)_n,(1_{G_n})_n)=0\right\rbrace\] is normal, so the metric ultraproduct $G_\ul$ is a topological group and the metric $\d_\ul$ is bi-invariant. For more on ultraproduct of groups, see \cite{Pestov2008}.
\end{rmk}

\subsection{Measure Spaces}

We will now define the ultraproduct of a sequence of probability spaces using Carath\'eodory's method. Let $\{(\fx_n,\rb_n,\mu_n)\}_{n\in\bn}$ be a family of probability spaces and let $\fx_\ul$ be their ultraproduct. We define 
\begin{gather*}
  \theta: \rp(\fx_\ul)\to [0,\infty],\\
\theta(A_\ul):=\inf\left\lbrace \sum_{i\in\bn} \lim_{n\in\ul} \mu_n(B_{n}^i) : A_\ul\subset \bigcup_{i\in\bn} [B_{n}^i]_\ul,\ B_{n}^i\in\rb_n\ \forall n,i\in\bn\right\rbrace.
\end{gather*}

\begin{prop} 
  The function $\theta$ defined above is an outer measure.
\end{prop}
\begin{proof}
For this, we have to check that $\theta(\emptyset)=0$, that if $A_\ul\subset C_\ul$ then $\theta(A_\ul)\leq \theta(C_\ul)$ and that for every sequence $\{A_\ul^j\subset \fx_\ul\}_j$ we have $\theta(\cup_j A_\ul^j)\leq \sum_j \theta(A_\ul^j)$. This can be done exactly as for the Lebesgue measure, see \cite[114D]{Fremlin1}.
  \begin{itemize}
  \item Since $\emptyset\subset [\emptyset]_\ul$, we must have that $\theta(\emptyset)=0$. 
  \item Suppose $A_\ul\subset C_\ul$. For every family $\{B_{n}^i\in\rb_n\}_{i,n}$ such that $C_\ul\subset \cup_i [B_{n}^i]_\ul$, we have that $A_\ul\subset \cup_i [B_{n}^i]_\ul$ so that $\theta(A_\ul)\leq \theta(C_\ul).$
  \item Let $\{A_{\ul}^j\}_{j\in\bn}$ be a sequence of subsets of $\fx_\ul$ and fix $\eps>0$. For every $j\in\bn$, fix a family $\{B_{n}^{j,i}\in\rb_n\}_{i,n}$ such that \[A_{\ul}^j \subset \bigcup_{i\in\bn} [B_{n}^{j,i}]_\ul\text{ and } \sum_{i\in\bn} \lim_{n\in\ul}\mu_n(B_{n}^{j,i})\leq \theta(A_{\ul}^j)+2^{-j}\eps.\] Then $\cup_j A_{\ul}^j\subset \cup_{j,i} [B_{n}^{j,i}]_\ul,$ so \[\theta(\cup_j A_\ul^j)\leq \sum_{i,j} \lim_{n\in\ul}\mu_n(B_{n}^{j,i})\leq \sum_j \theta(A_{\ul }^j)+\eps.\qedhere \] 
  \end{itemize}
\end{proof}

Whenever we have an outer measure, Carath\'eodory's theorem gives us a way of constructing a measure space. 

\begin{dfn}
  The \textbf{measure ultraproduct} of a family of probability spaces 
  $\{(\fx_n,\rb_n,\mu_n)\}_{n\in\bn}$ is the probability space $(\fx_\ul,\rb_\ul,\mu_\ul)$, where 
  \begin{align*}
    \rb_\ul:=&\{A_\ul\subset X_\ul : \theta(B_\ul)\geq \theta(B_\ul\cap A_\ul)+\theta(B_\ul\setminus A_\ul)\text{ for every }B_\ul\subset X_\ul\}\\
    \mu_\ul(A_\ul):=&\theta(A_\ul)\quad \text{ for every }A_\ul\in\rb_\ul.
  \end{align*}
\end{dfn}

Carath\'eodory's theorem, see for example \cite[113C]{Fremlin1}, tells us that $(X_\ul,\rb_\ul,\mu_\ul)$ is a measure space. In the following proposition we describe which subsets of the ultraproduct are measurable and we show how to compute their measure. 

\begin{prop} \label{prop:consmeas} Let $\{(\fx_n,\rb_n,\mu_n)\}_{n\in\bn}$ be a family of probability spaces and let $(X_\ul,\rb_\ul,\mu_\ul)$ be the measure space associated to $\theta$ via the Carath\'eodory's method, that is the measure ultraproduct of the family of probability spaces. 
  \begin{enumerate}
  \item For every sequence $\{A_n\in\rb_n\}_n$ we have $[A_n]_\ul\in\rb_\ul$ and $\mu_\ul([A_n]_\ul)=\lim_{n\in\ul}\mu_n(A_n).$
  \item For every $A_\ul\in\rb_\ul$ there is a sequence $\{B_n\in\rb_n\}_n$ such that $\mu_\ul(A_\ul\Delta [B_n]_\ul)=0$.
  \end{enumerate}
\end{prop}
\begin{proof}
  $(1)$ Let us prove that for every family $\{A_n\in\rb_n\}_n$, we have that $[A_n]_\ul\in\rb_\ul$. Consider a subset $B_\ul\subset \fx_\ul$, a real number $\eps>0$ and a family $C_{n}^i\in\rb_n$ such that \[B_\ul\subset\cup_i [C_{n}^i]_\ul\text{ and }\sum_i \theta([C_{n}^i]_\ul)\leq \theta(B_\ul)+\eps.\]

  So we have
  \begin{align*}
    \theta(B_\ul\cap [A_n]_\ul)+\theta(B_\ul\setminus [A_n]_\ul)\leq &\theta(\cup_i ([C_n^i]_\ul\cap[A_n]_\ul))+\theta(\cup_i([C_n^i]_\ul\setminus [A_n]))\\
=&\theta(\cup_i [C_n^i \cap A_n]_\ul)+\theta(\cup_i[C_n^i \setminus A_n])\\
\leq&\sum_i\lim_{n\in\ul}\left(\mu_n(C_{n}^i\cap A_n)+\mu_n(C_{n}^i\setminus A_n)\right)\\=&\sum_i \lim_{n\in\ul}\mu_n(C_{n}^i)\\ \leq&\theta(B_u)+\eps.
  \end{align*}
  As $\eps$ is arbitrary, $[A_n]_\ul$ is $\mu_\ul$-measurable.

As we have observed before, given two subsets $[B^1_n]_\ul$ and $[B^2_n]_\ul$ of $X_\ul$, we have that $[B_n^1]_\ul\cup [B_n^2]_\ul=[B_n^1\cup B_n^2]_\ul$. We remark that the same property does not hold for countable unions but the following lemma shows that a similar property holds in the measurable setting. 

\begin{lem}\label{lem:mescil}
  For every countable family $\{B_n^i\in\rb_n\}_{i,n\in\bn}$ there is a family $\{C_n\in\rb_n\}_{n\in\bn}$ such that \[\cup_i[B_n^i]_\ul\subset [C_n]_\ul\ \text{ and }\  \lim_{n\in\ul}\mu_n(C_n)=\lim_{i\to\infty}\lim_{n\in\ul}\mu_n(\cup_{j=1}^iB_n^j).\] 
\end{lem}
\begin{proof}
The proof is a standard diagonal argument for ultraproducts. For every $n$ and $i$, we set $D_n^i:=\cup_{j=1}^i B_n^j$. For $i\geq 1$, put
\[L_i:=\left\lbrace m\in\{i,i+1,\ldots\}:\ \left|\lim_{n\in\ul}\mu_n(D_{n}^{i})-\mu_m(D_{m}^{i})\right|\leq \frac{1}{2^i}\right\rbrace.\]
Observe that $L_i\in \ul$. We define the function \[f:\bn\to \bn\quad\text{ as }f(n):=\left\lbrace\begin{array}{lc}\max\{i : n\in L_i\} &\text{for }n\in \cup_i L_i\\ 1&\text{ otherwise}
\end{array}\right.\]

By construction $f(n)\leq n$, $f(n)$ tends to infinity as $n\to\ul$ and, for every $m$ in a subset $I_0\in \ul$, we have $|\lim_{n\in\ul}\mu_n(D_{n}^{f(m)})-\mu_m(D_{m}^{f(m)})|\leq 2^{-f(m)}$. 

We set $C_n:=D_{n}^{f(n)}$. For every $i\in\bn$ and for every $n\in L_i$, we have that $f(n)\geq i$, hence $C_n\supset D_n^{i}$. Since this is true for every $i$, we obtain that $[C_n]_\ul\supset \cup_i [D_n^i]_\ul=\cup_i [B_n^i]_\ul$ which implies that  \[\lim_{n\in\ul}\mu_n(C_n)\geq\lim_{i\to\infty}\lim_{n\in\ul}\mu_n(\cup_{j=1}^iB_n^j).\]

 Finally \[\mu_m(C_m)=\mu_m(D_{m}^{f(m)})\leq \lim_{n\in\ul}\mu_n(D_{n}^{f(m)})+\frac{1}{2^{f(m)}}\leq \lim_{i\to\infty}\lim_{n\in\ul}\mu_n(\cup_{j=1}^iB_{n}^j)+\frac{1}{2^{f(m)}}.\qedhere\]
\end{proof}

Let us now compute the measure of $[A_n]_\ul$. By definition of $\theta$, we must have $\theta([A_n]_\ul)\leq \lim_\ul \mu_n(A_n)$. For the reverse inequality, fix $\eps>0$ and consider a countable family $\{B_n^i\in\rb_n\}_{i,n}$ such that \[\left|\theta([A_n]_\ul)-\sum_{i\in\bn} \lim_{n\in\ul} \mu_n(B_{n}^i)\right|<\eps.\]

By Lemma \ref{lem:mescil}, there is a family $\{C_n\in\rb_n\}_n$ such that $[C_n]_\ul\supset\cup_i[B_n^i]_\ul\supset [A_n]_\ul$ which satisfies 
\begin{align*}
  \lim_{n\in\ul}\mu_n(A_n)\leq\lim_{n\in\ul}\mu_n(C_n)\leq\lim_{i\to\infty}\lim_{n\in\ul}\mu_n(\cup_{j=1}^iB^j_n) \leq\sum_{j=0}^\infty\lim_{n\in\ul}\mu_n(B^j_n)\leq \theta([A_n]_\ul)+\eps.
\end{align*}
Since $\eps$ is arbitrary, we obtain that  $\theta([A_n]_\ul)= \lim_\ul \mu_n(A_n)$.

\medskip

$(2)$ Let $A_\ul\in\rb_\ul$ be a measurable subset.  By definition of $\theta$, for every $j\in\bn$, there is a countable family $\{B_n^{i,j}\in\rb_n\}_{n,i}$ such that \[A_\ul\subset\cup_i [B_n^{i,j}]_\ul\ \text{ and }\  \sum_{i\in\bn} \lim_\ul \mu_n(B_n^{i,j})-\mu_\ul(A_\ul)\leq 2^{-j}.\]

By Lemma \ref{lem:mescil}, for every $j\in\bn$, there is a family $\{C^j_n\in\rb_n\}_n$ such that \[A_\ul\subset \cup_i [B_n^{i,j}]_\ul\subset [C_n^j]_\ul\ \text{ and }\ \mu_\ul([C_n^j]_\ul)-\mu_\ul(A_\ul)\leq 2^{-j}.\]
Observe that \[\fx_\ul\setminus \bigcap_{j\in\bn} [C_{n}^j]_\ul=\bigcup_{j\in \bn}\fx_\ul\setminus [C_n^j]_\ul=\bigcup_{j\in\bn}[\fx_n\setminus C_n^j]_\ul,\] so again by Lemma \ref{lem:mescil}, there is a family $\{D_n\in \rb_n\}_n$ such that \[\mu_\ul\left([D_n]_\ul\Delta (\cup_j[\fx_n\setminus C_n^j]_\ul)\right)=0.\] Hence if we define $B_n:=\fx_n\setminus D_n$, we have $\mu_\ul([B_n]_\ul\Delta (\cap_j[ C_n^j]_\ul))=0$ and \[\mu_\ul(A_\ul\Delta [B_n]_\ul)\leq\lim_i \mu_\ul(\cap_{j=1}^i[C_n^j]_\ul\setminus A_\ul)=0.\qedhere\] 
\end{proof}

\begin{rmk}\label{rmk:measalgult}
  Proposition \ref{prop:consmeas} implies that the measure algebra of the ultraproduct of a family of probability spaces is the metric ultraproduct of their measure algebras. See \cite[Section 328]{Fremlin3}.
\end{rmk}

\subsection{Maharam-type}

We now prove that the ultraproduct of a family of finite or standard probability spaces is a nice, homogeneous probability space. The following theorem is a special case of \cite{Jin2000} (which is written in the language of non-standard analysis). 

\begin{thm}\label{thm:maharamultra}
  Let $\{(\fx_n,\rb_n,\mu_n)\}_n$ be a sequence of diffuse standard probability spaces or a sequence of finite spaces equipped with their uniform counting measure such that $\lim_{n\in\ul} |\fx_n|=\infty$. Then the measure ultraproduct $(\fx_\ul,\rb_\ul,\mu_\ul)$ is measurably isomorphic to $(\{0,1\}^\br,\nu^\br)$ where $\nu$ is the normalized counting measure on $\{0,1\}$ and $\nu^\br$ is the product measure. That is, the measure algebras $\malg(\fx_\ul,\mu_\ul)$ and $\malg(\{0,1\}^\br,\nu^{\br})$ are isomorphic. 
\end{thm}

Observe that $\fx_\ul$ and $\{0,1\}^\br$ are not isomorphic as sets: they do not have the same cardinality. To prove the theorem, we recall the notion of Maharam type, see \cite[331F]{Fremlin3}. 

\begin{dfn} Let $(\fx,\mu)$ be a probability space and let us denote by $\fa=\malg(\fx,\mu)$ its measure algebra. 
  \begin{itemize}
  \item A subset $\ca\subset \fa$ $\sigma$-\textbf{generates}, if $\fa$ is the smallest $\sigma$-subalgebra of $\fa$ containing $\ca$.
  \item The \textbf{Maharam type} of the measure algebra $\fa$ is the smallest cardinal of any subset of $\fa$ which $\sigma$-generates $\fa$.
  \item A measure algebra $\fa$ is \textbf{homogeneous} if the Maharam type of $\fa$ is equal to the Maharam type of $\malg(A,\mu/\mu(A))$ for every $A\in\fa$. 
  \end{itemize}
\end{dfn}

All the homogeneous probability measure algebras which have the same Maharam type are isomorphic, see \cite[331L]{Fremlin3}. 

\begin{thm}\label{thm:maharam}
  Every homogeneous probability measure algebra $\fa$ is isomorphic to the measure algebra of $(\{0,1\}^Z,\nu^Z)$ for a set $Z$ which has the cardinality of the Maharam type of $\fa$. 
\end{thm}

We can now prove the theorem.

\begin{proof}[Proof of Theorem \ref{thm:maharamultra}]
First observe that $\malg(\fx_\ul,\mu_\ul)$ has at most the cardinality of the continuum, because by Remark \ref{rmk:measalgult}, $\malg(\fx_\ul,\mu_\ul)$ is the metric ultraproduct of a family of separable metric spaces. So we have to show that the Maharam type of $\malg(A_\ul,\mu_\ul/\mu_\ul(A))$ is at least the continuum for every $A_\ul\subset \fx_\ul$ measurable and non negligible. 

We start showing the result when $(\fx_n,\rb_n,\mu_n)$ is a diffuse standard probability space for every $n$. By Proposition \ref{prop:consmeas}, there is a sequence $\{A_n\in\rb_n\}_n$ such that $A_\ul=[A_n]_\ul$ up to measure $0$. Since for every $n$, the measure space $(A_n,\rb_n\bigr|_{A_n},\mu_n/\mu_n(A_n))$ is also a standard probability space, it is enough to show that the Maharam type of $\malg(\fx_\ul,\mu_\ul)$ is at least the continuum. For this we will use the following standard result, which is proved in \cite[331J]{Fremlin3}.

\begin{lem}\label{lem:mahaind}
  Let $\fa$ be a measure algebra and let $Z$ be a set. Suppose that there is a family $\{A_z\}_{z\in Z}$ of measurable mutually independent sets of measure $\mu(A_z)=1/2$. Then the Maharam type of $\fa$ is greater or equal to the cardinality of $Z$. 
\end{lem}

We now exhibit a continuum family of independent sets of $(\fx_\ul,\rb_\ul,\mu_\ul)$. For every $n$, take a countable family $\cb_n =\{B_n^i\}_i$ of measurable mutually independent set of $\fx_n$ of measure $1/2$. For every function $f:\bn\to\bn$, put $B^f_\ul:=[B_n^{f(n)}]_\ul$.  If we denote with $\bn_\ul$ the ultraproduct of $\{\bn,\bn,\ldots\}$, then for every $f\in\bn_\ul$ the measurable subset $B^f_\ul$ is well-defined, since it does not depend on the values of $f$ on subsets outside $\ul$. Observe also that if $f_1,\ldots,f_k$  differ $\ul$-almost always, then $B_\ul^{f_1},\ldots,B_\ul^{f_k}$ are independent. Therefore the family $\{B^f_\ul\}_{f\in\bn_\ul}$ is a family of measurable mutually independent sets of measure $1/2$. The cardinality of this family is the cardinality of $\bn_\ul$ which is the continuum by Remark \ref{rmk:card}. Hence Lemma \ref{lem:mahaind} concludes the proof of the theorem in the diffuse case.

\vspace{0.1cm}

The same strategy works for finite uniform spaces. Suppose that for every $n$, the measure space $(\fx_n,\mu_n)$ is a finite uniform space and suppose that $\lim_{n\in\ul} |\fx_n|=\infty$. For every $A_\ul\in \rb_\ul$, by Proposition \ref{prop:consmeas}, there is a sequence $\{A_n\}_n$ such that $A_\ul=[A_n]_\ul$ up to measure $0$. Let us denote by $g:\bn\to\bn$ the function such that $2^{g(n)}\leq |A_n|\leq 2^{g(n)+1}$. For every $n$, consider $C_n\subset A_n$ a subset of $2^{g(n)}$-elements. Observe that $\lim_{n\in\ul}g(n)=\infty$ and $\mu_\ul([C_n]_\ul)\geq \mu_\ul(A_\ul)/2$. For every $n$, there is a family $\cb_n=\{B^1_n,\ldots,B^{g(n)}_n\}$ of mutually independent sets such that $|B_n^i|=|C_n|/2$ for every $n$ and $i\leq g(n)$. As before, for every function $f:\bn\to\bn$ such that $f(n)\leq g(n)$, we can define $B^f_\ul:=[B_n^{f(n)}]_\ul$. If we denote with $Z_\ul$ the ultraproduct of $Z_n=\{1,\ldots,g(n)\}$, then, as before, for every $f\in Z_\ul$ the subset is well defined $B^f_\ul$ and  if $f_1,\ldots,f_k$  differ $\ul$-almost always, then $B_\ul^{f_1},\ldots,B_\ul^{f_k}$ are independent. Hence the family $\{B^f_\ul\}_{f\in Z_\ul}$ is a family of measurable mutually independent sets. Again by Remark \ref{rmk:card}, the cardinality of $Z_\ul$ is the continuum, so Lemma \ref{lem:mahaind} implies that the Maharam type of $[C_n]_\ul$ is the continuum. Observe that the Maharam type is monotone under taking ideals \cite[331H(c)]{Fremlin3}, hence also the Maharam type of $\malg(A_\ul,\mu_\ul/\mu_\ul(A))$ is the continuum. So the proof theorem is concluded. 
\end{proof}

\subsection{Automorphisms}

Let $(\fx,\mu)$ be a probability space and let $\aut(\fx,\mu)$ be its group of measure preserving automorphisms.

  \begin{itemize}
  \item The \textbf{uniform topology} on $\aut(\fx,\mu)$ is the topology defined by the metric \[\delta(S,T):=\mu(\{x\in\fx : Tx\neq Sx\}).\]
  \item The \textbf{weak topology} on $\aut(\fx,\mu)$ is the topology for which $T_n$ tends to $T$ if \[\mu(T_n(A)\Delta T(A))\to 0,\quad \forall A\subset \fx\text{ measurable}.\]
  \end{itemize}

\begin{exm} 
  Let $X=\{1,\ldots,n\}$ and let $\mu_n$ be the normalized counting measure on $X$. The group $\aut(X,\mu_n)$ is the symmetric group over $n$ elements $S_n$. The uniform topology is induced by the metric \[\delta(\sigma,\tau)=\frac{1}{n}\left|\{i : \sigma(i)\neq \tau(i)\}\right|.\]

  The metric $\delta$ is also called the Hamming distance.
\end{exm}

\begin{prop}\label{prop:ultrainner}
  Let $\{(\fx_n,\mu_n)\}_{n\in\bn}$ be a family of probability spaces. Then the metric-ultraproduct of the family $\{(\aut(\fx_n,\mu_n),\delta_n)\}_n$ embeds isometrically in $(\aut(\fx_\ul,\mu_\ul),\delta_{\ul})$.
\end{prop}
\begin{proof}
  Set $G:=\prod_n \aut(\fx_n,\mu_n)$ and define \[T:G\to \aut(\fx_\ul)\ \text{ as }\  T(g_n)_n[x_n]_\ul:=[g_nx_n]_\ul.\]
 
Given $(g_n)_n$ and $(h_n)_n$ in $G$, we have \[\delta_\ul(T(g_n)_n,T(h_n)_n)=\lim_{n\in\ul}\mu_n(\{x\in\fx_n :\ g_nx\neq h_nx\})=\lim_{n\in\ul}\delta_n(g_n,h_n),\] hence $T$ factorizes to an isometry from the metric ultraproduct of $\{(\aut(\fx_n,\mu_n),\delta_n)\}_n$ to $\aut(\fx_\ul,\mu_\ul)$. 
\end{proof}

We observe that since elements of the ultraproduct of the groups $\{(\fx_n,\mu_n)\}_{n\in\bn}$ can not act ergodically on $(\fx_\ul,\mu_\ul)$, this embedding is not surjective. 

\section{Limit of actions}

In this section we will study measure preserving actions on general probability spaces under the point of view of weak containment. We will prove that any measure preserving action on a diffuse probability space is weakly equivalent to an action on a standard probability space. This will be the key tool for understanding ultraproducts of sequences of probability measure preserving actions of a countable group $G$. We will introduce a compact, metric topology on the space of weak equivalence classes of actions which is equivalent to the topology defined in \cite{Abert2011c}, a sequence of (classes of) actions converges if all its ultraproducts are weakly equivalent and in this case, the ultraproduct is the limit. 

 We will denote by $a,b$ and $c$ the probability measure preserving actions (pmp) of $G$ on probability spaces, denoted by $(\fx_a,\mu_a),(\fx_b,\mu_b)$ and $(\fx_c,\mu_c)$ (which will not be standard in general). We will denote by $\act_d(G)$ the set of the pmp actions of $G$ on a (fixed) standard diffuse probability space and with $\act_f(G)$ the set of actions of $G$ on the finite sets $\{1,\ldots,n\}$ for $n\in\bn$, which we equip with their counting measure. We set $\act(G):=\act_d(G)\sqcup \act_f(G)$.

\begin{dfn}
  Let $a$ be a pmp action of $G$ on the probability space $(\fx_a,\mu_a)$. An action $b$ of $G$ is a \textbf{factor} of $a$, denoted $b\sqsubseteq a$, if there is a $G$-invariant isometric embedding of $\sigma$-algebras $\malg(\fx_b,\mu_b)\hookrightarrow \malg(\fx_a,\mu_a)$. 
\end{dfn}

More concretely factors of $a$ are exactly the restriction of $a$ to $G$-invariant $\sigma$-subalgebras of $\malg(\fx_a,\mu_a)$. 

\begin{rmk}
  By Theorem 343B of \cite{Fremlin3}, if $b$ is a pmp action of $G$ on the standard Borel probability space $(\fx_b,\mu_b)$ and $b$ is a factor of $a$, then there is a $G$-invariant measure preserving map $\pi:\fx_a\to \fx_b$. However, we will never use this theorem. 
\end{rmk}

Let $(\fx,\mu)$ be a probability space. We denote by $\part_k(\fx)$ the set of partitions of $\fx$ with $k$ atoms and by $\part_f(\fx)$ the set of finite partitions of $\fx$ (in what follows, $f$ will never be a natural number). For $\alpha\in\part_f(\fx)$, we will denote by $|\alpha|$ the number of atoms of $\alpha$. 

Given a pmp action $a$ of $G$, a finite subset $F\subset G$ and a partition $\alpha\in\part_f(\fx_a)$, we set \[\coeff(a,F,\alpha):=(\mu(A^i\cap g A^j))_{i,j\leq |\alpha|, g\in F}.\]

Given two pmp actions of $G$ (on the probability spaces $(\fx_a,\mu_a),(\fx_b,\mu_b)$), a finite subset $F\subset G$ and two finite partitions $\alpha=\{A_1,\ldots,A_k\}\in\part_f(\fx_a)$ and $\beta =\{B_1,\ldots,B_k\}\in\part_{|\alpha|}(\fx_b)$, we put \[\|\coeff(a,F,\alpha)-\coeff(b,F,\beta)\|_1:=\sum_{i,j\leq |\alpha|}\sum_{f\in F}|\mu_a(A_i\cap f A_j)-\mu_b(B_i\cap fB_j)|.\]

The following definition is due to Kechris, \cite{Kechris2010}.

\begin{dfn}\label{dfn:wc}
Let $a,b$ two pmp actions of $G$. We say that $a$ is \textbf{weakly contained} in $b$, and we will write $a\prec b$, if  for every $\eps>0$, for every finite subset $F\subset G$ and for every finite partition $\alpha\in\part_f(\fx_a)$ there is $\beta\in\part_{|\alpha|}(\fx_b)$ such that \[\|\coeff(a,F,\alpha)-\coeff(b,F,\beta)\|_1\leq \eps.\] 
\end{dfn}

Two actions $a$ and $b$ are \textbf{weakly equivalent}, denoted by $a\sim b$, if $a\prec b$ and $b\prec a$.

\begin{dfn}
  The \textbf{weak topology} on $\act_d(G)$ is the weakest topology for which the following sets form a base of open neighborhoods of $a\in\act_d(G)$: \begin{gather*}\{b\in \act_d(G):\ \|\coeff(a,F,\alpha)-\coeff(b,F,\alpha)\|_1<\eps\}\\ \text{ for } \alpha\in\part_f(\fx_a),\ F\subset G\text{ finite and }\eps>0.\end{gather*}
\end{dfn}

For a standard probability space $(\fx,\mu)$, we have an injective map $\act_d(G)\hookrightarrow \aut(\fx,\mu)^G$. The weak topology of $\act_d(G)$ corresponds to the product topology of the weak topology of $\aut(\fx,\mu)$. 

\subsection{WC topology}

We now define a topology equivalent to the topology defined in \cite{Abert2011c}. This topology will play a central role in the understanding of ultraproducts of actions. 

\begin{dfn}
Given two pmp actions $a,b$ of $G$, a finite subset  $F\subset G$ and $k\in\bn$, we define
  \begin{align*}
    \d_{F,\alpha}(a,b):=&\inf_{\beta\in\part_k(\fx_b)} \|\coeff(a,F,\alpha)-\coeff(b,F,\beta)\|_1\quad \text{for every }\alpha\in\part_k(\fx_a),\\
    \d_{F,k}(a,b):=&\sup_{\alpha\in\part_k(\fx_a)}\d_{F,\alpha}(a,b).
  \end{align*}
Clearly $a\prec b$ if and only if for every finite subset $F\subset G$ and $k\in\bn$, we have $\d_{F,k}(a,b)=0$. 
\end{dfn}

\begin{rmk}\label{rmk:ref}
Given two partitions $\alpha$ and $\beta$ of the probability space $(\fx,\mu)$, we say that $\alpha$ \textit{refines} $\beta$ if each atom of $\beta$ is (up to measure $0$) a union of atoms of $\alpha$. For every pmp actions $a,b$ of $G$, for every finite subset $F\subset G$ and finite partitions $\alpha,\beta\in\part_f(\fx_a)$ \[\text{if }\alpha\text{ refines }\beta\text{ then }\d_{F,\alpha}(a,b)\geq\d_{F,\beta}(a,b).\]
\end{rmk}

\begin{rmk}\label{rmk:genpart}
  Let $a$ and $b$ be two pmp actions of $G$. Let
  $\alpha_n\in\part_f(\fx_a)$ be an increasing sequence of partitions such that
  the algebra generated by $\cup_n\alpha_n$ is dense in $\malg(\fx_a,\mu_a)$. Then $a\prec b$ if and only if for every $F\subset G$ and $n\in\bn$, we have $\d_{F,\alpha_n}(a,b)=0$.

In fact, we have to show that for every finite partition $\alpha\in\part_f(\fx_a)$ and finite subset $F\subset G$ we have $\d_{F,\alpha}(a,b)=0$. Once $\alpha$ and $F$ are fixed, for every $\eps>0$ there are $n\geq 0$ and a partition $\beta\in\part_{|\alpha|}(\fx_a)$ refined by $\alpha_n$ such that $\|\coeff(a,F,\alpha)-\coeff(a,F,\beta)\|_1<\eps$. So \begin{align*} \d_{F,\alpha}(a,b)=&\inf_{\gamma\in\part_{|\alpha|}(\fx_b)}\|\coeff(a,F,\alpha)-\coeff(b,F,\gamma)\|_1\\ \leq& \inf_{\gamma\in\part_{|\alpha|}(\fx_b)}\|\coeff(a,F,\beta)-\coeff(b,F,\gamma)\|_1+\eps\\ 
\leq& \d_{F,\beta}(a,b)+\eps\\
\leq& \d_{F,\alpha_n}(a,b)+\eps=\eps.\end{align*}
\end{rmk}

\begin{prop} \label{prop:triang}Given three pmp actions $a,b$ and $c$ of $G$ for every $\alpha\in\part_f(\fx_a)$, we have $\d_{F,\alpha}(a,c)\leq \d_{F,\alpha}(a,b)+\d_{F,|\alpha|}(b,c).$
\end{prop}
\begin{proof}
Put $k=|\alpha|$. The proof is a straightforward computation: 
\begin{align*}
\d_{F,\alpha}(a,c)=&\inf_{\gamma\in\part_k(\fx_c)} \|\coeff(a,F,\alpha)-\coeff(c,F,\gamma)\|_1\\
\leq & \inf_{\beta\in\part_k(\fx_b)}\inf_{\gamma\in\part_k(\fx_c)} \left(\|\coeff(a,F,\alpha)-\coeff(b,F,\beta)\|_1+\|\coeff(b,F,\beta)-\coeff(c,F,\gamma)\|_1\right)\\
\leq & \inf_{\beta\in\part_k(\fx_b)}\left(\|\coeff(a,F,\alpha)-\coeff(b,F,\beta)\|_1+\inf_{\gamma\in\part_k(\fx_c)}\|\coeff(b,F,\beta)-\coeff(c,F,\gamma)\|_1\right)\\
\leq & \d_{F,\alpha}(a,b)+\sup_{\beta\in\part_k(\fx_b)}\inf_{\gamma\in\part_k(\fx_c)} \|\coeff(b,F,\beta)-\coeff(c,F,\gamma)\|_1\\ 
\leq &\d_{F,\alpha}(a,b)+\d_{F,|\alpha|}(b,c).\qedhere
\end{align*}
\end{proof}

\begin{dfn}\label{dfn:wct}
  The \textbf{WC-topology} on $\act(G)$ is the topology generated by the family of pseudo-metrics $\bar\d_{F,k}(a,b):=\d_{F,k}(a,b)+\d_{F,k}(b,a)$, where $F\subset G$ is any finite subset and $k\in\bn$. 
\end{dfn}

The topology is not $T_1$ and two actions have the same closure if and only if they are weakly equivalent. We denote by $\bact(G)$ the space of weakly-equivalent classes of actions. The WC-topology descends to a metric topology on $\bact(G)$. The definition of the WC-topology is similar to the definition given by Burton in \cite{Burton2015}. In the same paper he proved that the topology is equivalent to the topology of \cite{Abert2011c}. We will give a simpler and different proof in Theorem \ref{thm:compact}. 

The following proposition will be crucial to understand limits for the WC-topology. 

\begin{prop}\label{prop:reductiona}
  Let $\{a,a_1,a_2,\ldots\}$ be a family of actions of $G$. Then for every finite subset $F\subset G$, the following conditions are equivalent
  \begin{enumerate}
  \item for every finite partition $\alpha\in\part_f(\fx_a)$, we have $\lim_n\d_{F,\alpha}(a,a_n)=0$,
  \item for every $k\in\bn$, we have $\lim_{n}\d_{F,k}(a,a_n)=0$. 
  \end{enumerate}
\end{prop}
\begin{proof}
  Condition $(2)$ is by definition stronger than condition $(1)$, so let us suppose that $(1)$ holds. Fix $\eps>0$. For $k\in\bn$ set  
\[C:=\left\lbrace\coeff(\beta,F,a):\ \beta\in\part_k(\fx_a)\right\rbrace\subseteq [0,1]^{|F|k^2}.\]
By compactness, there are partitions $\alpha_1,\ldots,\alpha_j\in\part_k(\fx_a)$ such that 
\[\forall x\in C\text{ there is }i\leq j\text{ such that }\|\coeff(\alpha_i,F,a)-x\|_1\leq \eps.\]

Consider the finite partition $\alpha$ generated by $\alpha_1,\ldots,\alpha_j$. By hypothesis there is $N\in\bn$ such that for every $n\geq N$, we have that $\d_{F,\alpha}(a,a_n)<\eps$. Since $\alpha$ refines $\alpha_i$ for every $i$, we also have that $\d_{F,\alpha_i}(a,a_n)<\eps$ for every $i\leq j$ and $n\geq N$. So for $n\geq N$ and for every $\beta\in \part_f(\fx_a)$, there is $i\leq j$ such that $\|\coeff(\beta,F,a)-\coeff(\alpha_i,F,a)\|_1\leq \eps$, therefore \[\d_{F,\beta}(a,a_n)\leq \|\coeff(\beta,F,a)-\coeff(\alpha_i,F,a)\|_1+\d_{F,\alpha_i}(a,a_n)\leq 2\eps.\qedhere\]
\end{proof}

The following proposition is inspired by Theorem 5.3 of \cite{Conley2013a}.

\begin{prop}
  For a sequence of actions $a_n\in \act_d(G)$, the following are equivalent:
  \begin{enumerate}
  \item for every finite subset $F\subset G$ and $\alpha\in\part_f(\fx_a)$, we have $\d_{F,\alpha}(a,a_n)\to 0$,
  \item there is a family of automorphisms $T_n\in\aut(\fx_{a_n})$ such that $T_na_nT_{n}^{-1}$ converges to the action $a$ in the weak topology.
  \end{enumerate}
\end{prop}
\begin{proof}
  The fact that $(2)$ implies $(1)$ follows directly from the definitions, so we can suppose that $(1)$ holds. By a diagonal argument, we can find an increasing sequence of finite partitions $(\alpha_n)_n=(\{A_1^n,\ldots,A_{k_n}^n\})_n$ and an increasing sequence of finite subsets $F_n$ of $G$ such that $\d_{F_n,\alpha_n}(a,a_n)$ tends to $0$, $\cup_n F_n=G$ and the algebra generated by $\cup_n \alpha_n$ is dense in $\malg(\fx,\mu)$. By $(1)$, there is a sequence of partitions $(\beta_n)_n=(\{B_1^n,\ldots,B_{k_n}^n\})_n$ such that $\|\coeff(a,F_n,\alpha_n)-\coeff(a_n,F_n,\beta_n)\|_1$ tends to $0$, which we can choose to satisfy $\mu(A^n_i)=\mu(B^n_i)$. For every $n$, there is $T_n\in\aut(\fx,\mu)$ such that $\alpha_n=T_n\beta_n$. Now observe that $\coeff(T_na_nT_n^{-1},F_n,T_n\beta_n)=\coeff(a_n,F_n,\beta_n)$, so $(2)$ holds. 
\end{proof}

The following corollary is well-known (in the standard setting). 

\begin{crl}\label{crl:three}
For every pmp action $b$ on any probability space, the set of $\{a\in\act_d(G):\ a\prec b\}$ is weakly closed. 
\end{crl}
\begin{proof}
  We use Proposition \ref{prop:triang}. Let $(a_n)_n$ be a sequence which converges weakly to $a$ such that $a_n\prec b$ for every $n$. By the (easy part of the) previous proposition, for every $\alpha\in\part_f(\fx_a)$ and $F\subset G$ finite, we have that $\d_{F,\alpha}(a,a_n)\to 0$. Hence \[\d_{F,\alpha}(a,b)\leq \d_{F,\alpha}(a,a_n)+d_{F,k}(a_n,b)=\d_{F,\alpha}(a,a_n)\to 0.\qedhere\]
\end{proof}

\begin{dfn}\label{dfn:fix}
  For every pmp action $a$ of $G$ and for every $g\in G$, we set
  \begin{align*}
    \fix_g(a):=&\{x\in \fx_a:\ gx=x\},\\
    |\fix_g(a)|:=&\mu_a(\fix_g(a)).
  \end{align*}
\end{dfn}

\begin{prop}
  For every $g\in G$, the map $|\fix_g(\cdot)|:\bact(G)\to [0,1]$ is well-defined and (WC-)continuous. 
\end{prop}
\begin{proof}
  Let $a,b\in \act(G)$. By Rokhlin lemma, for every $\eps>0$, there are $A_\eps,B_\eps\subset \fx_a$ and $N\geq 1$, such that $\alpha:=\{\fix_g(a),A_\eps,gA_\eps,\ldots,g^NA_\eps,B_\eps\}$ is a partition of $\fx_a$ and $\mu(B_\eps)\leq\eps$. Put $F:=\{1_G,g,\ldots,g^N\}$ and observe that if $\d_{F,\alpha}(a,b)\leq \eta$, then \[|\fix_g(b)|\leq |\fix_g(a)|+\eta+\eps.\qedhere\]
\end{proof} 

\subsection{Every action is weakly equivalent to a standard one}


\begin{thm}\label{thm:thma}
  Every pmp action $a$ of the countable group $G$ on a diffuse space has a standard factor which is weakly equivalent to $a$. In particular every pmp action of $G$ is weakly equivalent to an action on a standard Borel probability space. 
\end{thm}

We remark that the theorem was also essentially proved for ultraproduct actions in the proof of the main theorem of \cite{Abert2011c}. We start showing that any pmp actions has at least a diffuse standard factor. 

\begin{lem}\label{lem:difffact}
  Every pmp action $a$ of $G$ on a diffuse space has a standard diffuse factor. 
\end{lem}
\begin{proof}
  If $(\fx_a,\mu_a)$ does not have any atom, we can find an increasing sequence of finite partitions $(\alpha_n)_n\subset \part_f(\fx_a)$ such that the measure of each atom in $\alpha_n$ is less than $1/n$ for every $n$. Then observe that the $G$-invariant $\sigma$-algebra generated by $\cup_n G\alpha_n$ is a separable measure algebra without atoms, so the factor associated is a factor of $a$ on a diffuse, standard probability space. 
\end{proof}

The theorem follows from two facts: the weak topology on $\act_d(G)$ is separable and the following easy lemma. 

\begin{lem}\label{lem:crit}
  For two pmp actions $a$ and $b$, the following are equivalent. 
  \begin{enumerate}
  \item The action $a$ is weakly contained in $b$, $a\prec b$. 
  \item We have $\{c\in \act(G):\ c\prec a\}\subseteq \{c\in \act(G):\ c\prec b\}$. 
  \end{enumerate}

Moreover if $(\fx_a,\mu_a)$ does not have any atom, then we can take $c$ in $(2)$ to be in $\act_d(G)$. 
\end{lem}
\begin{proof}
  The fact that $(1)$ implies $(2)$ follows from the transitivity of the weak containment. For the converse take a finite partition $\alpha\in\part_f(\fx_a)$ and a finite subset $F\subset G$. The $\sigma$-closure of the $G$-invariant algebra generated by $\alpha$ is a factor of $a$ which we denote by $c\in\act(G)$. By construction $ \d_{F,\alpha}(a,c)=0$ and by $(2)$, we have $c\prec b$. So $\d_{F,\alpha}(a,b)\leq \d_{F,\alpha}(a,c)+\d_{F,|\alpha|}(c,b)=0$. For the moreover part, we can consider the factor $c'$ associated to the $\sigma$-closure of the $G$-invariant algebra generated by $\alpha$ and the standard factor constructed in Lemma \ref{lem:difffact}.
\end{proof}

\begin{proof}[Proof of Theorem \ref{thm:thma}]
  By Corollary \ref{crl:three}, the set $A:=\{c\in\act_d(G):c\prec a\}$ is weakly closed. Let $\{b_n\}_{n\in\bn}$ be a countable weakly-dense subset of $A$. For every $n$, let $\{\beta_n^k\}_{k\in\bn}$ be an increasing sequence of finite partitions of $\fx_{b_n}$ which generate the $\sigma$-algebra. Let $\{F_n\}_n$ be an increasing sequence of finite subsets of $G$. For every $n,m,k\in\bn$, let $\alpha_{n}^{k,m}$ be a partition of $\fx_a$ such that \[\|\coeff(b_n,F_m,\beta_n^k)-\coeff(a,F_m,\alpha_n^{k,m})\|_1\leq \frac{1}{m}.\]

Consider the $G$-invariant $\sigma$-algebra $\fa$ generated by the partitions $\{\alpha_n^{k,m}\}_{n,k,m}$. Then $\fa$ is separable, since it is generated by finite partitions and $G$ is countable, so the associated factor $b$ is a factor of $a$ on a standard diffuse probability space which by construction weakly contains $b_n$ for every $n$. Corollary \ref{crl:three} implies that 
\[\{c\in\act_d(G):c\prec a\}=\overline{\{b_n\}_n}\subseteq \{c\in \act_d(G):\ c<b\}\]
therefore $(2)$ of Lemma \ref{lem:crit} holds, hence $a\prec b$.
\end{proof}

From now on, we will identify $\bact(G)$ with the set of weak equivalence classes of actions of $G$ on any diffuse of finite uniform probability space. 

\subsection{Ultraproduct and weak equivalence}

Given a pmp action $a$ of $G$, a partition $\alpha\in\part_f(\fx_a)$ and a finite subset $F\subset G$ we denote by $\alpha_F$ the partition generated by the $F$-translates of $\alpha$.  

\begin{dfn}\label{dfn:quasihomo}
  Consider two pmp actions $a$ and $b$ of $G$ and let us fix a partition $\alpha\in\part_f(\fx_a)$, a finite subset $F\subset G$ and $\delta>0$. A $(\alpha,\delta,F)$-\textbf{homomorphism} $\ph$ from $a$ to $b$, is a homomorphism from the measure algebra of $\alpha_F$, to the measure algebra $\malg(\fx_b,\mu_b)$, which satisfies
  \begin{itemize}
  \item $\mu_b(f\ph(A)\Delta\ph(f(A)))<\delta$ for every $A\in\alpha$ and $f\in F$,
  \item $\sum_{A\in\alpha_F} |\mu_b(\ph(A))-\mu_a(A)|<\delta$. 
  \end{itemize}
\end{dfn}

We denote by $\hom(a,\alpha,F,\delta,b)$ the set of $(\alpha,\delta,F)$-homomorphisms from $a$ to $b$

\begin{prop}\label{prop:weakquasi}
  An action $a$ is weakly contained in $b$ if and only if for every $\alpha\in\part_f(\fx_a)$, for every finite subset $F\subset G$ and for every $\delta>0$, the set $\hom(a,\alpha,F,\delta,b)$ is not empty. 
\end{prop}
\begin{proof}
  Suppose that $a\prec b$. Given $\alpha\in\part_k(\fx_a)$, a finite subset $F\subset G$ which contains the identity and $\eps>0$, we consider $\alpha_F=\{A_1,\ldots,A_k\}$.  By hypothesis there is a partition $\beta=\{B_1,\ldots,B_k\}\in\part_k(\fx_b)$ such that $\|\coeff(a,F,\alpha_F)-\coeff(b,F,\beta)\|_1< \eps$. Set $\ph(A_i)=B_i$. Given $A\in \alpha$ and $f\in F$ there are $I,J\subset\{1,\ldots,k\}$ such that $A=\sqcup_{i\in I}A_i$ and $fA=\sqcup_{j\in J} A_j$. Then 
\begin{align*} 
\mu_b(f\ph(A)\Delta \ph(fA))=&\mu_b(\ph(A))+\mu_b(\ph(fA))-2\mu_b((f(\sqcup_{i\in I}B_i))\cap(\sqcup_{j\in J}B_j))\\ 
\leq&\mu_b(\sqcup_{i\in I}B_i)+\mu_b(\sqcup_{j\in J} B_j)-2\sum_{i\in I}\sum_{j\in J}\mu_b(fB_i\cap B_j)\\
\leq&\mu_a(\sqcup_{i\in I}A_i) +\mu_a(\sqcup_{j\in J} A_j)-2\sum_{i\in I}\sum_{j\in J}\mu_a(fA_i\cap A_j)+4\eps\\
\leq&2\mu_a(A)-2\mu_a(fA\cap fA)+4\eps=4\eps.
\end{align*}

For the converse fix $\alpha=\{A_1,\ldots,A_k\}\in\part_k(\fx_a)$, a finite subset $F\subset G$ which contains the identity and $\delta>0$. Take $\ph\in \hom(a,\alpha,F,\delta,b)$. Define $B_i=\ph(A_i)$ and $\beta=\ph(\alpha)$. For $i,j\in \{1,\ldots,k\}$ and $f\in F$, we have
 \begin{align*}
|\mu_a(A_i\cap fA_j)-\mu_b(B_i\cap fB_j)|=&|\mu_a(A_i\cap fA_j)-\mu_b(\ph(A_i)\cap f\ph(A_j))|\\
\leq&|\mu_a(A_i\cap fA_j)-\mu_b(\ph(A_i\cap f A_j))|+\delta\\
\leq&|\mu_a(A_i\cap fA_j)-\mu_a(A_i\cap f A_j)|+2\delta=2\delta.\qedhere
\end{align*}
\end{proof}

\begin{dfn}  \label{dfn:ultralimit}
  Let $(a_n)_n$ be a sequence of pmp actions of $G$. The \textbf{ultraproduct} of the sequence $(a_n)_n$ is the action of $G$ on the ultraproduct measure space of the sequence $\{(\fx_{a_n},\mu_{a_n})\}_n$ given by $g[x_n]_\ul:=[gx_n]_\ul$, see Proposition \ref{prop:ultrainner}.
\end{dfn}

\begin{prop}[Theorem 5.3 of \cite{Conley2013a}]\label{prop:weakfact}
  Let $a\in\act(G)$, let $(b_n)_n$ be a sequence of actions of $G$ and let $b_\ul$ be its ultraproduct. Then $a\prec b_\ul$ if and only if $a\sqsubseteq b_\ul$. 
\end{prop}
\begin{proof}
  Let us suppose that $a\prec b_\ul$. Let $(\alpha_n)_n$ be an increasing sequence of partitions of $\fx_a$, such that the algebra generated by them $\ca$ is a dense $G$-invariant subalgebra of $\malg(\fx_a,\mu_a)$. Let $F_n\subset G$ be an increasing sequence of finite subsets which contain the identity and such that $\cup_n F_n=G$. By Proposition \ref{prop:weakquasi}, for every $n$ we can take $\ph_n\in \hom(a,\alpha_n,F_n,1/n,b)$. We denote by $\ph_\ul:\ca\to\malg(\fx_\ul,\mu_\ul)$ the map defined by $\ph_\ul(A):=[\ph_n(A)]_\ul$. It is clear that $\ph_\ul$ is a $G$-invariant homomorphism which respect the measure, hence it is an isometry with respect to the natural metric on $\malg$. Therefore we can extend $\ph_\ul$ to a $G$-invariant isometric embedding of $\sigma$-algebras $\malg(\fx_a,\mu_a)$ to $\malg(\fx_\ul,\mu_\ul)$. 
\end{proof}

\subsection{WC-compactness}

We now show that the ultraproduct of a sequence of actions defined in Definition \ref{dfn:ultralimit} is the limit with respect to the ultrafilter $\ul$ for the WC-topology. Observe that the ultraproduct of a sequence of actions always exists, so Theorem \ref{thm:compact} implies that the topology is sequentially compact. Since the topology is metrizable, the topology is also compact, so we obtain Theorem 1 of \cite{Abert2011c}. On the other hand, the theorem characterizes the convergence of sequences in terms of ultraproducts of actions. Since the same characterization holds for the topology in \cite{Abert2011c}, the two topology are equivalent.

\begin{thm}\label{thm:compact}
For every sequence of actions $(a_n)_n\subset \bact(G)$ the $\ul$-WC-limit of the sequence exists and is weakly equivalent to $a_\ul$. In particular a sequence $(a_n)_n$ WC-converges to $a$ if and only if $a$ is weakly equivalent to the ultraproduct action $a_\ul$ with respect to every ultrafilter $\ul$.
\end{thm}
\begin{proof}
  Let $(a_n)_n$ be a sequence in $\bact(G)$. By Theorem \ref{thm:thma}, and a little abuse of notations, we have that the ultraproduct of the sequence $a_\ul$ is an element of $\bact(G)$. We want to show that the $\ul$-WC-limit of the sequence $(a_n)_n$ is $a_\ul$. By Proposition \ref{prop:reductiona}, we have to show that for every $\alpha\in\part_f(\fx_{a_\ul})$, for every finite subset $F\subset G$ and $k\in\bn$, we have \[\lim_{n\in\ul}\d_{F,k}(a_n,a_\ul)=\lim_{n\in\ul}\d_{F,\alpha}(a_\ul,a_n)=0.\]

  For every finite partition $\alpha_\ul=\{[A_n^1]_\ul,\ldots,[A^k_n]_\ul\}\in\part_k(\fx_{a_\ul})$, consider the family of partitions $\alpha_n:=\{A_n^1,\ldots,A^k_n\}\in \part_f(\fx_{a_n})$. Then for every finite subset $F\subset G$, we have \[\lim_{n\in\ul}\|\coeff(a_n,F,\alpha_n)-\coeff(a_\ul,F,\alpha_\ul)\|_1=0,\] and hence $\lim_{n\in\ul}\d_{F,\alpha_\ul}(a_\ul,a_n)=0$. On the other hand, suppose that there are a finite subset $F\subset G$, an integer $k\in\bn$ and $\eps>0$ such that $\lim_{n\in\ul}\d_{F,k}(a_n,a_\ul)>\eps$. Then for every $n$ in a set $I\in\ul$, there is a partition $\alpha_n=\{A_n^1,\ldots,A_n^k\}\in\part_k(\fx_{a_n})$ such that $\d_{F,\alpha_n}(a_n,a_\ul)>\eps$. So if we take the partition $\alpha_\ul:=\{[A_n^1]_\ul,\ldots,[A_n^k]_\ul\}$ we observe that \[\lim_{n\in\ul}\|\coeff(a_n,F,\alpha_n)-\coeff(a_\ul,F,\alpha_\ul)\|_1\geq \eps,\] which is a contradiction. 
\end{proof}

The following interesting corollary was remarked in both \cite[Corollary 3.1]{Abert2011c} and \cite[Proposition 5.7]{Conley2013a}. 

\begin{crl}
  Let $a$ be a pmp action of $G$ and let $b\in\act(G)$ be an action which is weakly contained in $a$. Then there is an action $a'$ weakly equivalent to $a$ such that $b$ is a factor of $a'$.
\end{crl}

Increasing sequences of actions always admit limits and such limits are easily described. 

\begin{prop}\label{prop:upward}
  Let $(a_n)_n$ be an upward directed sequence of actions in $\bact(G)$.
  \begin{enumerate}[(1)]
  \item The sequence converges to an action $a\in\bact(G)$.
  \item For every $n\in\bn$, we have $a_n\prec a$.
  \item If $b\in\bact(G)$ satisfies that $a_n\prec b$ for every $n\in\bn$, then $a\prec b$. 
  \end{enumerate}
\end{prop}
\begin{proof}
  By compactness, there is a WC-converging subsequence $(a_{n_k})_k$ and let $a$ be its limit. We claim that $(2)$ and $(3)$ holds for $a$. For this fix $n>0$, a finite subset $F\subset G$ and $\alpha\in\part_f(\fx_{a_n})$. Since $(a_{n_k})_k$ WC-converges to $a$, \[\d_{F,\alpha}(a_n,a)\leq \d_{F,\alpha}(a_n,a_{n_k})+\d_{F,|\alpha|}(a_{n_k},a)=\d_{F,|\alpha|}(a_{n_k},a)\stackrel{k}{\to} 0\] 
hence $a_n\prec a$ for every $n$. Let $b\in\bact(G)$ an action such that $a_n\prec b$ for every $n\in\bn$. Then for every partition $\alpha\in\part_f(\fx_a)$ and finite subset $F\subset G$,  \[\d_{F,\alpha}(a,b)\leq \d_{F,\alpha}(a,a_{n_k})+\d_{F,|\alpha|}(a_{n_k},b)=\d_{F,\alpha}(a,a_{n_k}) \stackrel{k}{\to} 0,\] hence $\d_{F,\alpha}(a,b)=0$ for every $\alpha$ and $F$, which implies $a\prec b$. 

Let $a'$ and $a''$ two different cluster points of $(a_n)_n$. Then by $(2)$ we have that $a_n\prec a'$ and $a_n\prec a''$ for every $n$ and by $(3)$ we get that $a'\prec a''$ and $a''\prec a'$, that is $a'$ is weakly equivalent to $a''$ and hence they represent the same element of $\bact(G)$. 
\end{proof}


\begin{crl}\label{crl:profvsultra}
  Let $(a_n)_n$ be an increasing sequence of finite actions and let $a$ be the associated profinite action. Then $(a_n)_n$ WC-converges to $a$. In particular the profinite action $a$ is weakly equivalent to the ultraproduct action $a_\ul$. 
\end{crl}
\begin{proof}
  By Proposition \ref{prop:upward}, it is enough to show that for every action $b\in\bact(G)$ such that $a_n\prec b$ for every $n$, we have that $a\prec b$. Fix such an action $b$. For every $n$, we denote by $\alpha_n\in\part_f(\fx_a)$ the partition on clopen sets such that $a\bigr|_{\alpha_n}=a_n$. By Remark \ref{rmk:genpart}, it is enough to show that for every finite subset $F\subset G$ and $n\in\bn$, we have $\d_{F,\alpha_n}(a,b)=0$. This is straightforward 
\[\d_{F,\alpha_n}(a,b)\leq \d_{F,\alpha_n}(a,a_n)+\d_{F,|\alpha_n|}(a_n,b)=0.\qedhere\]
\end{proof}

\section{Sofic entropy}


In this section we will show that for free groups and $\PSL_k(\bz)$ the sofic entropy of profinite actions depends on the sofic approximation.

\subsection{Sofic actions}

Let $G$ be a countable group, let $\mathbf F$ be a countable free group and let $\pi:\mathbf F\to G$ be a surjective homomorphism. Let us fix a section $\rho:G\to \mathbf F$ which maps the identity to the identity. Given any action $a$ of $G$, we denote by $a^{\mathbf F}$ the action of $\mathbf{F}$ defined by $a^{\mathbf F}(g):=a(\pi(g))$. For an action $a$, recall that $|\fix_g(a)|$ is the measure of the fixed point of $g$, (Definition \ref{dfn:fix}). 

\begin{dfn}\label{dfn:soficapr}
  A \textbf{sofic approximation} $\Sigma=(a_n)_n$ of $G$ is a sequence of finite actions $a_n\in\act_f({\mathbf F})$ such that
  \begin{itemize}
  \item for every $g\in \ker\pi$, we have that $\lim_n|\fix_g(a_n)|= 1$,
  \item for every $g\notin \ker\pi$, we have that $\lim_n|\fix_g(a_n)|= 0$.
  \end{itemize}
\end{dfn}

A group is \textit{sofic} if it has a sofic approximation. 

\begin{dfn}\label{dfn:soficact}
  Given a sofic approximation $\Sigma=(a_n)_n$ of $G$, the ultraproduct action $a_\ul$ of the sequence $(a_n)$ is an action of ${\mathbf F}$ for which $\ker\pi$ acts trivially. Hence we can see the action $a_\ul$ as a $G$-action, which we will denote by $a^\Sigma_\ul$ and we will call it the \textbf{sofic action} associated to $\Sigma$. 
\end{dfn}

\begin{dfn}\label{dfn:actsof}
  An action $a$ of the group $G$ is \textbf{sofic} if there exists a sequence of finite actions $(a_n)_n\subset\act_f({\mathbf F})$ such that
  \begin{itemize}
  \item for every $\alpha\in\part_f(\fx_a)$ and $F\subset G$ finite, we have $\lim_n\d_{\rho(F),\alpha}(a^{\mathbf F},a_n)=0$,
  \item for every $g\in\ker\pi$, we have $\lim_n|\fix_g(a_n)|=1$.
  \end{itemize}
\end{dfn}

We observe that the definition does not depend on the choice of $\rho$. Moreover we could also ask that $\d_{F,\alpha}(a^{\mathbf F},a_n)\to 0$ for every finite subset $F$ of the free group ${\mathbf F}$. Observe also that if an action $a$ of $G$ is sofic, then the sequence $(a_n)_n$ as in Definition \ref{dfn:actsof} is a sofic approximation, so any group which admits a sofic free action is sofic. 

\begin{prop}\label{prop:soficfactor}
   An action $a\in\act(G)$ of the countable group $G$ is sofic if and only if there is a sofic approximation $\Sigma$ of $G$ such that $a$ is a factor of the sofic action $a_\ul^\Sigma$. 
\end{prop}
\begin{proof}
  If the action $a$ is sofic, then by construction $a^{\mathbf F}$ is weakly contained in $a_\ul$ and hence by Proposition \ref{prop:weakfact}, we have that $a$ is a factor of $a^\Sigma_\ul$. On the other hand, if $a$ is a factor of $a^\Sigma_\ul$, then $\d_{\rho(F),\alpha}(a^{\mathbf F},a_n)=\d_{\rho(F),\alpha}(a_\ul,a_n),$
which tends to zero by Theorem \ref{thm:compact}. 
\end{proof}

\begin{rmk}
  It is not known whether every sofic action of a sofic group is of the form $a^\Sigma_\ul$ for a sofic approximation $\Sigma$ of $G$. The question is even open for Bernoulli shifts. They are sofic by \cite{Elek2010}, but we do not know if there exists a non-amenable group for which the Bernoulli shift is of the form $a^\Sigma_\ul$. 
\end{rmk}

One can show that Definition \ref{dfn:actsof} is equivalent to the definition of Elek and Lippner \cite{Elek2010} in terms of colored graphs and to the (unpublished) definition of Ozawa of soficity of pseudo full groups, see Definition 10.1 in \cite{Conley2013a}. Remark that the authors in \cite{Conley2013a} prove that Definition \ref{dfn:actsof} implies the soficity of the pseudo full group in the proof of Theorem 10.7. 

\subsection{Sofic entropy}

In what follows, we use the definitions and notations of Kerr \cite{Kerr2013a} with the only exception that we will use ultralimits instead of limsup in the definition. 

Let $G$ be a countable sofic group and let $a\in\act(G)$ be an action of $G$ on a standard probability space. Let ${\mathbf F}$ be a free group, let $\pi:{\mathbf F}\to G$ be a surjective homomorphism and let $\rho:G\to {\mathbf F}$ be a section of $\pi$ which maps the identity to the identity. Fix a sofic approximation $\Sigma=(a_n)_n$ as in Definition \ref{dfn:soficapr}. Consider two partitions $\xi\leq\alpha\in\part_f(\fx_a)$, a finite subset $F\subset G$ and $\delta>0$. We put $\hom(a,\alpha,F,\delta,a_n):=\hom(a^{{\mathbf F}},\alpha,\rho(F),\delta,a_n)$ (see Definition \ref{dfn:quasihomo}), where $a^{{\mathbf F}}(g)=a(\pi(g))$. We denote by $|\hom(a,\alpha,F,\delta,a_n)|_\xi$ the cardinality of the set of $(\alpha,\delta,\rho(F))$-homomorphisms from $a$ to $a_n$ restricted to $\xi$, as explained in \cite{Kerr2013a}.

We can now define the entropy of $a$ with respect to $\Sigma$, as follows
\begin{align*} 
  \heb_\Sigma^\xi(\alpha,F,\delta,a):=&\lim_{n\in\ul} \frac{1}{|\fx_{a_n}|}\log\left(\vert\hom(a,\alpha,F,\delta,a_n)\vert_\xi\right),\\
  \heb_\Sigma^\xi(\alpha,F,a):=&\inf_{\delta>0}\heb_\Sigma^\xi(\alpha,F,\delta,a),\\
  \heb_\Sigma^\xi(\alpha,a):=&\inf_{F\subset G} \heb_\Sigma^\xi(\alpha,F,a),\\
  \heb_\Sigma^\xi(a):=&\inf_{\alpha>\xi} \heb_\Sigma^\xi(\alpha,a),\\
  \heb_\Sigma(a):=&\sup_{\xi} \heb^\xi_\Sigma(a).
\end{align*}
where $\xi$ and $\alpha$ are finite partitions of $\fx_a$ with $\xi<\alpha$, $F\subset G$ is a finite subset and $\delta>0$ is a real number. Observe that the definition does not depend on the section $\rho:G\to {\mathbf F}$, since for every $g\in \ker\pi$, we have that $\lim_n|\fix_g(a_n)|= 1$. If for some $\alpha,\delta,F$ and $n$ the set $\hom(\alpha,F,\delta,a_n)$ is empty, we will set $\heb^\xi_\Sigma(\alpha,F,\delta,a)=-\infty$.
 
\begin{prop}\label{prop:entrweakcon}
Let $G$ be a countable sofic group and let $a\in\act(G)$ be an action of $G$. Fix a sofic approximation $\Sigma$ and let $a^\Sigma_\ul$ be the sofic action as in Definition \ref{dfn:soficact}. Then $\heb_\Sigma(a)>-\infty$ if and only if $a\prec a^\Sigma_\ul$. 
\end{prop}

This proposition is a corollary of Proposition \ref{prop:weakquasi}. We observe that it is also a special case of Proposition 6 of \cite{Graham2014}.

\begin{proof}
  Let ${\mathbf F}$ be a free group, let $\pi:{\mathbf F}\to G$ be a surjective homomorphism, let $\rho:G\to{\mathbf F}$ be a section and let $\Sigma=(a_n)_n$ be a sofic approximation. 

Suppose that $\heb_\Sigma(a)>-\infty$. Then there is a finite partition $\xi\in\part_f(\fx_a)$ such that for every $\alpha\in\part_f(\fx_a)$ with $\alpha>\xi$, for every finite subset $F\subset G$ and for every $\delta>0$, we have 
 \[\left\lbrace n\in\bn:\ \hom(a^{{\mathbf F}},\alpha,\rho(F),\delta,a_n)\neq\emptyset \right\rbrace\in\ul.\]
 Take $\ph_n\in \hom(a^{{\mathbf F}},\alpha,\rho(F),\delta,a_n)$ and define $\ph_\ul(A):=[\ph_n(A)]_\ul$. By construction we have that $\ph_\ul\in \hom(a,\alpha,F,\delta,a^\Sigma_\ul)$ and hence the set is not empty. Therefore Proposition \ref{prop:weakquasi} implies that $a\prec a^\Sigma_\ul$. 

Conversely, if we suppose that $a\prec a^\Sigma_\ul$, Proposition \ref{prop:weakquasi} tells us that for every finite partition $\alpha=\{A^1,\ldots,A^k\}\in\part_f(\fx_a)$, for every finite subset $F\subset G$ and for every $\delta>0$ the set $\hom(a,\alpha,F,\delta,a^\Sigma_\ul)$ is not empty. Take an element $\ph_\ul\in \hom(a^{\mathbf F},\alpha,\rho(F),a_\ul)$. Choose a family of subsets $\{B_n^i\}_{i,n}$ such that $\ph_\ul(A^i)=[B^i_n]_\ul$ and set $\ph_n(A^i):=B_n^i$. Then, we observe that for every $\eps>0$, the set of $n\in\bn$ such that $\ph_n\in \hom(a,\alpha,F,\delta+\eps,a_n)$ is in $\ul$, hence $\heb_\Sigma(a)>-\infty$.
\end{proof} 

 Let $G$ be a residually finite group and let $(H_n)_n$ be a chain of finite index subgroups of $G$. We denote by $a^{(H_n)}$ the profinite action associated to the sequence which we will always assume to be free. If the profinite action $a^{(H_n)}$ is free, then the sequence of finite actions gives us a sofic approximation of the group which we will denote by $\Sigma_{(H_n)}$.

Combining Proposition \ref{prop:entrweakcon} with Corollary \ref{crl:profvsultra}, we get the following interesting result. 

\begin{crl}\label{crl:entropyprofinite}
  Let $G$ be a residually finite group and let $(H_n)_n$ be a chain of finite index subgroups of $G$ such that the associated profinite action is free. Then for every action $a\in\act(G)$ we have that $\heb_{\Sigma_{(H_n)}}(a) >-\infty$ if and only if $a\prec a^{(H_n)}$. 
\end{crl}

Since the corollary holds for every ultrafilter, it is still true for the usual definition of entropy with $\limsup$. In particular the sofic entropy of a non-strongly ergodic action with respect to a sofic approximation given by expanders is always $-\infty$.

\begin{crl}
  Let $G$ be a residually finite group let $(K_n)_n$ be a chain of finite index subgroups of $G$ which has property $(\tau)$. For every non-strongly ergodic action $a$ of $G$, we have $\heb_{\Sigma_{(K_n)}}(a)=-\infty$.
\end{crl}
\begin{proof}
  It is enough to observe that if $(K_n)_n$ has property $(\tau)$, then $a^{(K_n)}$ is strongly ergodic, as explained for example in Lemma 2.2 of \cite{Abert2012b}, and an action weakly contained in a strongly ergodic action is also strongly ergodic (cf. Lemma 5.1 \cite{Abert2012b}). 
\end{proof}

\subsection{Sofic entropy of profinite actions}

Combining Corollary \ref{crl:entropyprofinite} with \cite{Abert2012b}, we can now show that for some groups sofic entropy of profinite actions crucially depends on the sofic approximation.  

\begin{thm}\label{thm:entrprofcont}
 Let $G$ be a countable free group or $\PSL_k(\bz)$ for $k\geq 2$. Then there is a continuum of normal chains $\{(H^r_n)_n\}_{r\in\br}$ such that $\heb_{\Sigma_{(H^r_n)}}(a^{(H^s_n)})>-\infty$ if and only if $r=s$. 
\end{thm}

Note that the sofic entropy of profinite actions is either $0$ or $-\infty$ as shown in Section 4 of \cite{Chung2014}, see also Lemma \ref{lem:genprof}. Theorem \ref{thm:entrprofcont} follows from Corollary \ref{crl:entropyprofinite} and the following theorem.

\begin{thm}[Ab\'{e}rt-Elek, \cite{Abert2012b}]\label{thm:AbertElek}
   Let $G$ be a countable free group or $\PSL_k(\bz)$ for $k\geq 2$. Then there is a continuum of normal chains $\{(H^r_n)_n\}_{r\in\br}$ such that $a^{(H^s_n)}\prec a^{(H^r_n)}$ if and only if $r=s$. 
\end{thm}
\begin{proof}[Sketch of the Proof for $G=\PSL_k(\bz)$, $k\geq 3$.]
   Let $(H_n)_n$ be the sequence of congruence subgroups of $G$, so that the family $\{G/H_n\}$ is a family of pairwise-non isomorphic finite non Abelian simple groups. For $I=\{i_1,i_2,i_3,\ldots\}\subset \bn$ infinite, we denote by $a^I$ the profinite action associated to the normal chain $(\cap_{i\leq n} H_i)_n$. Observe that for an infinite $I$, the profinite action $a^I$ is free and moreover, by property $(T)$, it is strongly ergodic. Therefore we can apply Lemma 5.2 of \cite{Abert2012b} to get that $a^I\prec a^J$ if and only if $I\subset J$. So if we take any continuum of incomparable infinite subset of $\bn$, then the associated profinite actions $\{a^I\}_I$ are weakly incomparable. 

In order to see that Theorem \ref{thm:AbertElek} holds for $\PSL_2(\bz)$ and for free groups, we can use that the congruence subgroups in $\PSL_2(\bz)$ have property $(\tau)$ and that the proof above passes to finite index subgroups, see the proof of Theorem 3 in \cite{Abert2012b}.  
\end{proof}

\begin{rmk}
  Theorem \ref{thm:AbertElek} holds for a large variety of groups. In fact the Strong Approximation Property claims that any Zariski dense subgroup of the rational point of a rational algebraic linear group, has infinitely many pairwise non-isomorphic simple non-Abelian finite quotients, see \cite[Window 9]{Lubotzky2003}. This was used in \cite{Abert2012b} to find family of pairwise inequivalent free actions of linear property $(T)$ groups. One can then combine this fact with Margulis normal subgroup theorem to show that Theorem \ref{thm:entrprofcont} holds for many lattices of higher rank algebraic linear groups.
\end{rmk}

We know give an example of an action which has positive entropy with respect to a sofic approximation and $-\infty$ with respect to another. We will do this considering the examples of Theorem \ref{thm:entrprofcont} and taking the diagonal product with respect to a Bernoulli shift. Then Bowen's computation for such actions will allow us to conclude. 

\begin{thm}\label{thm:randinf}
  Let $G$ be a countable free group or $\PSL_k(\bz)$ for $k\geq 2$. For every $r\geq 0$, there is an action $a$ of $G$ and two sofic approximations $\Sigma_1$ and $\Sigma_2$ such that $\heb_{\Sigma_1}(a)=r$ and $\heb_{\Sigma_2}(a)=-\infty$. 
\end{thm}

In the proof of the theorem, we will need the following easy lemma which was point out to us by L. Bowen. 

\begin{lem}\label{lem:genprof}
  Let $(H_n)_n$ be a chain of finite index subgroups of $G$ and denote by $a=a^{(H_n)}$ the associated profinite action of $G$. For every $\eps>0$, there is a generating partition $\alpha$ of $\fx_a$ with $H(\alpha)\leq \eps$. 
\end{lem}
\begin{proof}
  Put $i_0:=[G:H_1]$ and $i_n:=[H_{n}:H_{n+1}]$, without lost of generality we can suppose that $i_n\geq 2$ for every $n\geq 0$. Let us fix $\eps>0$ and take $N\in\bn$ such that $2^{-(N-1)}+\sum_{n\geq N}n2^{-(n-1)}<\eps$. For every $n\geq N$, take a clopen $A_n\subset \fx_n$ such that 
  \begin{itemize}
  \item $A_n\cap A_m=\emptyset$ if $n\neq m$,
  \item $A_n$ is a clopen set associated to a conjugate $H^g_n$ of $H_n$, that is it has measure $1/[G:H_n]$ and it is $H^g_n$-invariant.
  \end{itemize}

  Set $A_0:=\fx\setminus \cup_{n\geq N} A_n$ and $\alpha:=\{A_0,A_N,A_{N+1},\ldots\}$. The partition $\alpha$ is generating and observe that $\mu(A_n)\leq 2^{-n}$ and $\mu(A_0)\geq 1-2^{-(N-1)}$. We now compute the entropy of $\alpha$, 
  \begin{align*}
    H(\alpha)=&-\mu(A_0)\log(\mu(A_0))-\sum_{n\geq N} \mu(A_n)\log(\mu(A_n))\\
    \leq&-\log(1-2^{-(N-1)}) +\sum_{n\geq N}\frac{\log(i_1\ldots i_n)}{i_1\ldots i_n}\\
    \leq&2^{-(N-1)}+\sum_{n\geq N} 2^{-(n-1)}\sum_{j=1}^n \frac{\log(i_j)}{i_j}\\
    \leq&2^{-(N-1)}+\sum_{n\geq N}n2^{-(n-1)}<\eps.\qedhere
  \end{align*}
\end{proof}

\begin{proof}[Proof of Theorem \ref{thm:randinf}]
Let $(\fx,\mu)$ be a finite probability space with $H(\mu)=r$ and denote by $b$ the Bernoulli shift of $G$ on $(\fx^G,\mu^G)$. By Theorem \ref{thm:AbertElek}, there are two normal chains of finite index subgroups $(H_n)_n$ and $(K_n)_n$ such that the actions $a^{(H_n)}$ and $a^{(K_n)}$ are weakly incomparable and so the diagonal action $a^{(H_n)}\times b$ is not weakly contained in $a^{(K_n)}$. Lemma \ref{lem:genprof} and Bowen's Theorem \cite[Theorem 8.1]{Bowen2010} tell us that $\heb_{\Sigma_{(H_n)}}(a^{(H_n)}\times b)=H(\mu)=r$ and by Corollary \ref{crl:entropyprofinite} we have that $\heb_{\Sigma_{(K_n)}}(a^{(H_n)}\times b)=-\infty$.
\end{proof}


\bibliographystyle{plainnat}
\bibliography{/home/feaver/Dropbox/Documents/Isittrue/Bib}
\end{document}